\documentclass[twoside]{article}
\usepackage[preprint]{aistats2027}
\usepackage[round]{natbib}

\usepackage{microtype}
\usepackage{graphicx}
\usepackage{subcaption}
\usepackage{booktabs} 
\usepackage{enumitem}

\newcommand{\E}{\mathbb{E}}

\newcommand{\indep}{\perp \!\!\! \perp}
\newcommand{\Pb}{\mathbb{P}}

\newcommand{\obs}{\mathrm{obs}}
\newcommand{\argmin}{\mathop{\mathrm{arg \, min}}}

\newcommand{\ind}{\mathds{1}}

\usepackage{dsfont} 

\newcommand{\mis}{\mathrm{mis}}

\global\long\def\esp{\mathbb{E}}%
\global\long\def\F{\mathcal{F}}%
\global\long\def\R{\mathbb{R}}%

\usepackage{amsmath}
\usepackage{amssymb}
\usepackage{mathtools}
\usepackage{amsthm}
\usepackage{empheq}

\usepackage{hyperref}
\usepackage[capitalize,noabbrev]{cleveref}
\crefname{assumption}{assumption}{assumptions}
\Crefname{assumption}{Assumption}{Assumptions}
\crefname{propo}{proposition}{propositions}
\Crefname{propo}{Proposition}{Propositions}

\theoremstyle{plain}
\newtheorem{theorem}{Theorem}[section]
\newtheorem{propo}[theorem]{Proposition}
\newtheorem{proposition}[theorem]{Proposition}
\newtheorem{lemma}[theorem]{Lemma}

\theoremstyle{definition}

\newtheorem{assumption}{Assumption}
\theoremstyle{remark}
\newtheorem{remark}[theorem]{Remark}

\usepackage{xcolor}

\begin{document}

%

%

\twocolumn[

\aistatstitle{Handling Covariate Mismatch in Collaborative Linear Prediction}

\aistatsauthor{Alexis Ayme$^\star$ \And R\'emi Khellaf$^\star$}

\aistatsaddress{D\'epartement d’Informatique \\ École Normale Supérieure \\ PSL \& CNRS \And  Inria, PreMeDICaL Team,\\Universit\'e de Montpellier} ]

\begingroup
\renewcommand{\thefootnote}{}
\footnotetext{$^\star$ Equal contribution.}
\footnotetext{Correspondence to: \texttt{remi.khellaf@inria.fr}.}
\endgroup

\begin{abstract}
Training predictive models across multiple centers typically assumes that all centers collect the same set of covariates. In practice, however, they may record different features of their observations, a setting we refer to as \emph{covariate mismatch}. We study linear prediction under this challenging setting, assuming center-wise MCAR missingness patterns, and develop estimators that exploit information across centers despite heterogeneous feature sets. In the low-dimensional regime, we propose a plug-in estimator of the oracle linear predictor based on component-wise aggregation of covariance and cross-moment estimates. In higher dimensions, we study an impute-then-regress strategy that first completes the missing covariates using an exchangeability-preserving imputation procedure and then fits a ridge-regularized linear model. All proposed estimators are compatible with federated learning constraints: individual-level data remain local to each center, and only aggregated quantities are exchanged. We provide asymptotic and finite-sample learning rates for our predictors, explicitly characterizing their behaviour with the global dimension, the center-specific feature partition, and the distribution of samples across centers, and validate our approach through numerical experiments.
\end{abstract}


\section{Introduction}

The federated learning framework \citep{kairouz2021advances} enables institutions to train predictive models collaboratively by exchanging summary information without sharing raw individual-level data. In practice, multicenter collaborations often lack harmonized data collection. Participating centers frequently measure distinct sets of variables, use different instruments, or apply varying terminology for identical features. For instance, collaborating hospital cohorts typically collect data from distinct individuals, while recording varying clinical and laboratory features. This reality induces a ``covariate mismatch'': local datasets share a prediction target but observe different feature spaces. Statistically, this formulates a structured missing-data problem \citep{RUBIN76} where missingness is blockwise across centers rather than independent and identically distributed (i.i.d.) across entities. Consequently, classical missing-data literature, which focuses on entrywise patterns or homogeneous missing mechanisms, fails to capture the constraints and opportunities of federated data silos.

In this work, we establish a statistical learning framework for multicenter learning under blockwise missingness and federated constraints. We characterize when and how cross-center correlations can be leveraged to recover predictive signal under unobserved coordinates, and we derive rates that make explicit the roles of ambient dimension, center-specific feature partitions, and sample allocation across centers. The high-dimensional regime is particularly relevant: as federated datasets scale, imperfect harmonization across measurement pipelines inflates the nominal dimension while inducing strong feature correlations. This suggests that ridge-type analyses, which depend on the effective rather than the nominal dimension of the data, offer an informative lens for understanding model performance in modern federated learning.

\textbf{Our Contributions.} We introduce a theoretical and practical framework for learning under covariate mismatch. Our contributions are:
\begin{itemize}[leftmargin=*]
    \item \textbf{Setting formalization.} In \Cref{sec:setting}, we formalize collaborative learning under covariate mismatch. We introduce a definition of global risk, and frame the objective as estimating a single optimal linear predictor that adapts to local missingness patterns.
    \item \textbf{Plug-in method.} In \Cref{subsec:PlugIn}, we propose a consistent one-shot predictor that aggregates local second moments in low-dimensional settings. This method achieves zero-shot generalization on mismatched, unlabeled centers unseen during training.
    \item \textbf{Impute-then-regress method.} In \Cref{sec:ITR}, we propose and study Ridge-regularized impute-then-regress schemes for both low and high dimensional settings. We prove a distribution-free finite-sample bound (\Cref{thm:imputeechangeable}), valid for any exchangeable imputation rule. We analyze their practical federated implementation and communication costs.
    \item \textbf{Collaborative vs. local learning.} In \Cref{sec:Comparison}, we establish conditions under which collaborative strategies outperform local learning. By comparing their bias and variances, we demonstrate that collaboration reduces the effective dimension in high-dimensional and fragmented regimes.
    \item \textbf{Numerical experiments.} All theoretical claims are illustrated by reproducible simulations.
\end{itemize}


\textbf{Related Work.}
Prediction with missing features has been studied extensively, including approaches where the predictor depends on the missingness pattern \citep{lemorvana,le2020neumiss,ayme2022near} in the linear setting. The "Impute-then-Regress" strategy is a standard alternative, proven to be consistent for prediction tasks \citep{josse2019consistency, bertsimas2021beyond,le2021sa}. Imputation methods vary widely in complexity, ranging from classical iterative equations (MICE) \citep{van2011mice} to advanced approaches using Optimal Transport \citep{muzellec2020missing} or Generative Adversarial Networks (GANs) \citep{yoon2018gain, ipsen2022deal}. Recent works have extended these generative methods to federated settings \citep{balelli2023fed, yu2024communication}. From a statistical perspective, high-dimensional learning with missing values has been analyzed primarily in homogeneous settings, where missingness is unstructured (e.g., entry-wise Bernoulli) rather than block-wise across centers \citep{lounici2014high, agarwal2019robustness,sportisse2020debiasing, ayme2023naive, ayme2024random, park2025overparametrized}. Our theoretical framework builds upon the extensive literature on random design linear prediction \citep{tsybakov2003optimal,gyorfi2006distribution, bach2013non} and Ridge regression \citep{hsu2012random, dieuleveut2017harder}, leveraging recent finite-sample characterizations of generalization error to handle the federated mismatch \citep{richards2021asymptotics, mourtada2022elementary}.


\textbf{Notations.} For any positive integer $n$, let $[n] \coloneqq \{1, \dots, n\}$. We generally index data samples by $i$ and feature dimensions by $j$. Given a vector $x$ and a subset of indices $\mathcal{J}$, we denote by $x_{\mathcal{J}}$ the sub-vector formed by the entries of $x$ indexed by $\mathcal{J}$. For matrices $A$ and $B$ of compatible dimensions, $A \odot B$ denotes their Hadamard (element-wise) product ($\oslash$ the division). We use the Mahalanobis notation $\Vert u\Vert_m^2= u^\top Mu$. 

\section{Background}\label{sec:setting}
We consider a multicenter learning setting with $K$ centers communicating via a central server. Each center $k \in [K]$ holds a local dataset of size $n_k$, contributing to a total sample size $n = \sum_{k=1}^K n_k$. We tackle the challenging setting of \emph{covariate mismatch}, where centers systematically observe a specific subset of features, denoted by an index set $\obs(k)\subset [d]$, where $[d]=\cup_{k=1}^K \obs(k)$ is the dimension of the concatenated data, and $\mis(k):=[d]\setminus \obs(k)$ the set of unobserved features. For each sample $i$ collected at center $k$, the observed data are $\left(H_i, {X}_{i,\obs(H_i)},Y_i\right)$ where $H_i\in[K]$ denotes the center index, $X_i\in\mathbb{R}^d$ is the complete feature vector and $X_{i,\obs(k)}\in\mathbb{R}^{|\obs(k)|}$ is its restriction to the observed features at center $k$, and $Y_i\in\mathbb{R}$ the outcome.  
We assume that the training samples $\{(H_i, X_{i, \obs(H_i)}, Y_i)\}_{i=1}^n$ are i.i.d. draws from a distribution $\mathcal{P}$. For notational simplicity, we denote a generic test observation drawn from $\mathcal{P}$ (independent of the training set) as $(H, [X]_{\obs(H)}, Y)$.
Finally, we define the population second moment matrix $\Sigma := \E[XX^\top] \in \mathbb{R}^{d \times d}$ and the cross-moment vector $\gamma := \E[XY] \in \mathbb{R}^d$.

\begin{assumption}[Homogeneity/MCAR]\label{ass:mcar_iid}
    $H\indep (X,Y)$
\end{assumption}


\Cref{ass:mcar_iid} can be read in two ways: first, the joint distribution of $(X, Y)$ is shared across all centers, and second, the missingness mask (linked to $H$) operates as a block-wise Missing Completely At Random (MCAR) mechanism. Because collaborative learning under covariate mismatch remains unexplored, establishing a foundation under MCAR is a necessary first step. In the centralized missing data literature, this assumption is required for numerous methods, and negative minimax results demonstrate that learning can be impossible without it \citep{le2020linear, ayme2022near}.
Under \Cref{ass:mcar_iid}, the population moments $(\gamma, \Sigma)$ are shared across centers and are identifiable. Consequently, this framework rules out both covariate shift ($X \not\indep H$) and concept shift ($Y \not\indep H \mid X$).

The local and global risks are
\begin{align*}
&R^{(k)}(f^{(k)})=\E[(Y-f^{(k)}(X_{\obs(k)}))^2\mid H=k],\\
&R(f)\! :=\! \E \!\!\left[ \!\left( Y - f^{(H)}(X_{\obs(H)}) \!\right)^2 \right]\! \!=\! \!\sum_{k=1}^K \rho_k R^{(k)}(f^{(k)}),
\end{align*}
where $\rho_k = \Pb(H=k)$ is the center proportion.

\subsection{Optimal linear predictor}

Let the best linear predictor on the complete covariates be 
$$\theta_\star = \arg\min_{\theta\in\mathbb R^d}\E\big[(Y-X^\top\theta)^2\big] = \Sigma^{-1}\gamma,$$ 
assuming invertible $\Sigma$. Defining the residual as $\varepsilon = Y - X^\top\theta_\star$, we have $\E[\varepsilon X]=0$.
Importantly, even if the global relationship is well-specified ($\E[\varepsilon \mid X]=0$), the conditional mean on a feature subset is not generally linear: nonlinear correlations between observed and unobserved covariates can lead $\E[Y \mid X_{\obs(k)}]$ to be nonlinear in $X_{\obs(k)}$. Consequently, we characterize each center $k$ by its best linear approximation, defining its local predictor as the $L_2$ projection of $Y$ onto the linear span of its observed features, $X_{\obs(k)}$.

\textbf{Linear class and optimal risks.}
Let $\mathcal{F}_{\mathrm{lin}}$ denote the class of center-wise linear functions. Each $f \in \mathcal{F}_{\mathrm{lin}}$ is parameterized by local vectors $\{\theta^{(k)}\}_{k=1}^K$, predicting $f(x, k) := x^\top \theta^{(k)}$ for center $k \in [K]$ and local features $x \in \mathbb{R}^{\vert{}\obs(k)\vert{}}$. We denote by $\mathcal{F}^{(k)}_{\mathrm{lin}}$ the set of linear predictors for center $k$.


\begin{propo}\label{prop:best_linear}
Under \Cref{ass:mcar_iid}, the best local predictor $\theta_\star^{(k)} \in \mathcal{F}^{(k)}_{\mathrm{lin}}$ is the solution to the local least-squares problems and admits two equivalent representations
\begin{empheq}[left={\theta_\star^{(k)} \!\!=\!\! \empheqlbrace}]{align}
    & (\theta_{\star})_{\obs(k)}+ \Sigma_{\obs(k)}^{-1}\Sigma_{\obs(k),\mis(k)}(\theta_{\star})_{\mis(k)} \label{eq:theta_k_theta} \\
    & \Sigma_{\obs(k)}^{-1}\gamma_{\obs (k)}, \label{eq:theta_k_gamma}
\end{empheq}
where $\Sigma_{\obs(k)} := \Sigma_{\obs(k),\obs(k)}$. The optimal global risk is $R_\star(\mathcal{F}_{\mathrm{lin}}) = \sum_{k=1}^K \rho_k R_\star^{(k)}(\mathcal{F}^{(k)}_{\mathrm{lin}})$, with the minimal local risks given by
\begin{equation*}
    R_\star^{(k)}(\mathcal{F}^{(k)}_{\mathrm{lin}}) = \sigma^2 + \left\Vert (\theta_{\star})_{\mis(k)} \right\Vert_{V_k}^2,
\end{equation*}
where $\sigma^2 = \esp[\varepsilon^2]$ and $V_k$ is the Schur complement
\begin{equation*}
    V_k = \Sigma_{\mis(k)} - \Sigma_{\mis(k),\obs(k)}\Sigma_{\obs(k)}^{-1}\Sigma_{\obs(k),\mis(k)}.
\end{equation*}
\end{propo}

Proposition~\ref{prop:best_linear} (proved in \Cref{proof:best_linear}) characterizes each center's optimal linear predictor based solely on its observed coordinates $X_{\obs(k)}$, alongside the corresponding oracle risk. Unlike related literature that often assumes well-specified Gaussian linear models to derive convenient closed forms \citep{le2020linear,ayme2022near,cheng2023collaboratively}, this result is \emph{distribution-free}, requiring neither a correctly specified model for $Y\mid X$ nor Gaussian $X$.


\begin{remark}
Proposition~\ref{prop:best_linear} holds for singular $\Sigma$ (and hence $\Sigma_{\obs(k)}$) by replacing matrix inverses with Moore-Penrose pseudoinverses.
\end{remark}

\begin{remark}
This result extends to heterogeneous covariate distributions across centers by replacing $\Sigma$ with $\Sigma^{(k)}=\E[XX^\top\mid H=k]$. However, estimating $K$ distinct covariance matrices changes the federated problem and reduces the statistical benefits of collaboration.
\end{remark}

\textbf{Objective.} We seek a federated procedure that, despite covariate mismatch, outputs a predictor $\widehat f\in\mathcal{F}_{\mathrm{lin}}$ whose local risks $R^{(k)}(\widehat f^{(k)})$ approach $R_\star^{(k)}(\mathcal{F}^{(k)}_{\mathrm{lin}})$, ensuring the global risk $R(\widehat f)$ approaches $R_\star(\mathcal{F}_{\mathrm{lin}})$. A central goal is \emph{information pooling}: leveraging cross-center data to recover the predictive signal of unobserved coordinates $X_{\mis(k)}$ through their correlations with locally observed features.

\subsection{Plug-in methods}\label{subsec:PlugIn}

We now introduce methods for linear prediction in the low-dimensional asymptotic regime, deferring high-dimensionality to \Cref{sec:ITR}. \citet{cheng2023collaboratively} reconstruct $\theta_\star$ using inverse variance weighting of local OLS estimators in the spirit of \eqref{eq:theta_k_theta}, requiring local sample sizes $n_k > d$ and a known covariance $\Sigma$ from auxiliary data. In contrast, we propose a \emph{plug-in} strategy \citep{LohWainwright12} that estimates the global moments \((\Sigma,\gamma)\) directly from the pooled data, relaxing the sample size requirement to only $n > d$. To evaluate the closed-form oracle \eqref{eq:theta_k_gamma}, we crop these federated estimators to the coordinates observed at center \(k\):
\begin{equation}\label{eq:crop_estimator}
\widehat\theta^{(k)}_{\mathrm{PI}}
\;\coloneqq\;
\big(\widehat\Sigma_{\obs(k)}\big)^{-1}\widehat\gamma_{\obs(k)}.
\end{equation}
The resulting predictor deployed at center \(k\) is $\widehat f^{(k)}_{\mathrm{PI}}(x)
\;=\;
x^\top \widehat\theta^{(k)}_{\mathrm{PI}}$. 


To evaluate \eqref{eq:crop_estimator}, we propose two estimation strategies for $(\Sigma,\gamma)$ that compatible with the federated constraint.

\textbf{Zero-Imputed Estimators.} A naive approach imputes missing features with zeros, yielding global estimators from the average of local zero-filled moments:
\begin{equation*}
    \begin{cases}
    \displaystyle\hat\gamma^{\mathcal{I}^0}= \frac{1}{n}\sum_{i=1}^n  (M_i\odot X_i) Y_i= \sum_{k=1}^K \frac{n_k}{n} \hat\gamma_k^{\mathcal{I}^0} \\
    \displaystyle\hat\Sigma^{\mathcal{I}^0}= \frac{1}{n}\sum_{i=1}^n (M_i\odot X_i)(M_i\odot X_i)^\top= \sum_{k=1}^K \frac{n_k}{n}\hat\Sigma_k^{\mathcal{I}^0} ,
    \end{cases}
\end{equation*}
where $M_i \in \{0,1\}^d$  is the binary observation mask, i.e. $M_{ij}=\mathds{1}\{j\in\obs(H_i)\}$. While communication-efficient, this estimator is biased. The missing values act as a multiplicative dampening factor, yielding $\E [\hat\gamma^{\mathcal{I}^0}]= \mathrm{diag}(\Pi) \odot \gamma$ and $\E [\hat\Sigma^{\mathcal{I}^0}]= \Pi \odot \Sigma$, where $\Pi = \E[MM^\top]$ is the co-observation probability matrix.

\textbf{Debiased Estimators.} To correct this bias, we apply inverse probability weighting. We rescale the zero-imputed moments element-wise by the empirical co-observation frequencies $\hat{\Pi}_{jj'} = \frac{1}{n} \sum_{i=1}^n (M_i)_j (M_i)_{j'}$. This yields the Component-Wise (CW) estimators:
\begin{equation*}
\begin{cases}
    \hat\gamma^{\mathrm{cw}} = \hat\gamma^{\mathcal{I}^0} \oslash \mathrm{diag}(\hat \Pi) \\
    \hat\Sigma^{\mathrm{cw}} = \hat\Sigma^{\mathcal{I}^0 } \oslash \hat \Pi,
    \end{cases}
\end{equation*}
where $\oslash$ denotes element-wise division. Intuitively, $\hat\Sigma^{\mathrm{cw}}_{jj'}$ is the empirical covariance computed using only the samples observing both feature $j$ and $j'$.

\begin{theorem}[Consistency]\label{thm:consistency}
Assume the co-observation probability satisfies $\Pi_{jj'} > 0$ for all pairs $j,j' \in [d]$. Then, the plug-in predictor $\hat{f}_{\mathrm{PI}}$ constructed from $(\hat\Sigma^{\mathrm{cw}},\hat\gamma^{\mathrm{cw}})$ is consistent:
\begin{equation*}
    \hat\theta^{(k)}_{\mathrm{PI}} \xrightarrow[n \to \infty]{a.s.} \theta^{(k)}_\star.
\end{equation*}
Furthermore, for any new center $K+1$ whose observed features are supported by the federation (i.e., $\Pi_{jj'} > 0$ for all $j,j' \in \obs(K+1)$), the local predictor is also consistent: $\hat\theta^{(K+1)}_{\mathrm{PI}} \xrightarrow[n \to \infty]{a.s.} \theta^{(K+1)}_\star$.
\end{theorem}

This theorem (proved in \Cref{proof:thm_consistency}) shows that identifiability relies on the union of centers rather than their intersection. The global covariance $\Sigma$ can be reconstructed provided every feature pair is observed sufficiently often across the centers. Consequently, we can instantiate the optimal local predictor for a new center with an unseen observation pattern $\obs(K+1)$, as long as its feature correlations can be globally inferred. To our knowledge, this zero-shot generalization under covariate mismatch represents a novel capability in collaborative and federated learning.

\begin{remark}[Finite-Sample Upper Bound]
    In \Cref{prop:non-ass}, we extend this analysis to the finite-sample regime. Considering a modified version of the plug-in predictor, we establish a convergence rate for the excess risk $R^{(K+1)}(\hat f_{\mathrm{PI}}^{(K+1)})$ of order $O(\sqrt{d^2/n})$.
\end{remark}

\begin{remark}
    Plug-in methods support asynchronous, one-shot estimation. Unlike iterative optimization procedures such as distributed gradient descent, each center transmits its local sufficient statistics only once. Once the global moments $(\hat{\Sigma}, \hat{\gamma})$ are aggregated and redistributed, centers compute their specific predictors $\hat{\theta}^{(k)}_{\mathrm{PI}}$ locally without additional communication.
\end{remark}

While theoretically appealing, component-wise plug-in estimators are rarely deployed in practice. They suffer from finite-sample instability especially as the dimension increases, where the empirical covariance matrices become singular. To overcome these limitations, we now turn to alternative strategies tailored for the high-dimensional regime in \Cref{sec:ITR}.
\section{Impute-then-Regress (High-Dimensional Regime)}\label{sec:ITR}

Having studied predictors requiring a data-rich regime ($n \gg d$), we now shift to the high-dimensional setting ($d/n \to \eta > 0$). We focus on an Impute-then-Regress (ItR) strategy, which aligns with standard practice for missing data and admits stable, finite-sample, distribution-free guarantees. This approach first imputes missing entries, then applies Ridge regression to the completed data to mitigate the high dimensionality.

\subsection{Exchangeable Imputations}
We define an \emph{imputation procedure} $\mathcal{I}$ as any method that fills missing values while preserving observed entries, i.e., for a center $k$, the imputed vector $\tilde{X} = \mathcal{I}(X_{\obs(k)}, k) \in \mathbb{R}^d$ must satisfy $\tilde{X}_j = X_j$ if $j \in \obs(k)$.
\begin{assumption}[Exchangeability]\label{ass:exchangeable} 
    The joint distribution of $(\widetilde X,\widetilde X_1, \dots, \widetilde X_{n})$ is invariant under permutation of the indices.
\end{assumption}
Because data-dependent imputation relies on the entire dataset to fill missing entries, it breaks the i.i.d. property of the completed samples.  \Cref{ass:exchangeable} ensures mathematical tractability despite this dependence.
This assumption is standard in the missing data literature and has recently been leveraged in the context of conformal prediction \citep{zaffran2023conformal}. Popular imputation methods satisfy this assumption (e.g., zero-imputation, mean-imputation), as do many iterative procedures such as MICE \citep{van2011mice} or MissForest \citep{stekhoven2012missforest}, provided they are implemented in a permutation-invariant manner (i.e., not dependent on the specific ordering of the training data). In contrast, sequential approaches such as online imputation violate it.
In the following, we denote the risk of the best linear approximation of $Y$ given $\widetilde X$ as
$$
R_\star(\mathcal F_{\mathcal I})
\;\coloneqq\;
\inf_{\theta\in\mathbb R^d}\E\big[(Y-\theta^\top \widetilde X)^2\big],
$$
with $F_{\mathcal I}$ the set of linear predictors after data imputation by procedure $\mathcal I$.


\textbf{Optimal Linear Imputation.}
We introduce $\mathcal I^{\mathrm{opt}}$, the \emph{optimal linear} imputation, defined as
\begin{equation}\label{eq:optImpute}
    \widetilde{X}_{i,\mis(H_i)}
\;\coloneqq\;
\Sigma_{\mis(H_i),\obs(H_i)}\,\Sigma_{\obs(H_i)}^{-1} X_{i,\obs(H_i)}.
\end{equation}
This imputation is optimal in the sense that it provides the \emph{best linear approximation} of the missing features $X_{\mis(k)}$ given the observed features $X_{\obs(k)}$ (which corresponds to the conditional expectation when $X$ is a Gaussian vector). 
\begin{proposition} \label{prop:OptimalImputationApprox}
    Under \Cref{ass:mcar_iid}, the optimal risk over the class of linear functions on data completed by $\mathcal{I}^{\mathrm{opt}}$ matches the global optimal linear risk:
    \begin{equation*}
        R_\star(\mathcal F_{\mathcal{I}^{\mathrm{opt}}}) = R_\star(\mathcal F_{\mathrm{lin}}).
    \end{equation*}
\end{proposition}
\Cref{prop:OptimalImputationApprox} establishes that the performing an ItR procedure with the optimal linear imputation attains the optimal global linear risk defined in \Cref{eq:integratedrisk}.

Importantly, $\mathcal I^\mathrm{opt}$ is not meant to recover the true missing values (and thus distinct from the ``oracle'' imputation function), it plugs in their best linear prediction given observed features. 


\subsection{ItR with Ridge Regression}

Define the following moments of the imputed data
\[
\widehat\Sigma^{\mathcal I}=\frac1n\sum_{i=1}^n \widetilde X_i\widetilde X_i^\top,
\qquad
\widehat\gamma^{\mathcal I}=\frac1n\sum_{i=1}^n \widetilde X_iY_i.
\]
We study Ridge regression on the imputed sample \((\widetilde X_i,Y_i)_{i=1}^n\). For $\lambda\geq0$, the Ridge estimator is
\begin{equation}\label{eqdef:ridge_estimator}
    \widehat\theta_\lambda=(\widehat\Sigma^{\mathcal I}+\lambda I_d)^{-1}\widehat\gamma^{\mathcal I}.
\end{equation}

\textbf{Finite-sample bound of ItR.}
Following the analysis of \citet{mourtada2023local}, we rely on a bounded-outcome assumption to establish finite-sample guarantees for $\widehat\theta_\lambda$ that hold without assuming a specific covariate distribution or a well-specified linear model.
\begin{assumption}[Bounded outcome and covariance matrix]\label{ass:Ybound}
There exists \(m>0\) such that \(|Y|\le m\) almost surely and $\E \widetilde X \widetilde X^\top= \Sigma^{\mathcal{I}}$ exists. 
\end{assumption}
Assumption~\ref{ass:Ybound} enables the following distribution-free guarantees as $m$ only depends on the distribution of $Y$ and thus does not scale with $d$.



We define the final predictor using the truncation operator $T_m(t) \coloneqq \max(-m, \min(t, m))$:
\begin{equation*}
    \widehat f_\lambda(\widetilde X) \coloneqq T_m(\widehat\theta_\lambda^\top \widetilde X).
\end{equation*}

\begin{theorem}\label{thm:imputeechangeable}
Grant Assumptions~\ref{ass:mcar_iid} and \ref{ass:Ybound}, and let $\mathcal I$ be any exchangeable imputation procedure  (Assumption~\ref{ass:exchangeable}).
Then, the excess risk over the best linear predictor on the imputed data satisfies for any $\lambda\geq 0$
\[
\E\!\left[(Y-\widehat f_\lambda(\widetilde X))^2\right] - R_\star(\mathcal F_{\mathcal I})
\;\le\;
B_\lambda^{\mathcal I} \;+\; \frac{8m^2}{n}\, d_\lambda^{\mathcal I},
\]
where the bias induced by Ridge regularization is
\[
B_\lambda^{\mathcal I}
\;=\;
\inf_{\theta\in\mathbb R^d}\Big\{\E[(Y-\theta^\top \widetilde X)^2]-R_\star(\mathcal F_{\mathcal I})+\lambda\|\theta\|_2^2\Big\},
\] 
and
$
d_\lambda^{\mathcal I} = \mathrm{Tr}\!\Big(\Sigma^{\mathcal I}\big(\Sigma^{\mathcal I}+\lambda I_d\big)^{-1}\Big)$
is the effective dimension of the covariance of the imputed covariates, also known as \emph{degrees of freedom}, with $\Sigma^{\mathcal I}=\E[\widetilde X\widetilde X^\top]$.
\end{theorem}
Theorem~\ref{thm:imputeechangeable} (proved in \Cref{proof:thm_imputeechangeable}) states that the excess risk decomposes into (i) a bias term \(B_\lambda^{\mathcal I}\), and (ii) a variance term of order \(\tfrac{1}{n} d_\lambda^{\mathcal I}\), where the effective dimension \(d_\lambda^{\mathcal I}\) depends on the covariance \(\Sigma^{\mathcal I}\) of the imputed features. In particular, \(d_\lambda^{\mathcal I}\) quantifies the degree of freedom of the imputed design: it decreases when \(\widetilde X\) becomes more correlated (i.e., more concentrated in a low-dimensional subspace), reflecting that stronger correlations reduce statistical complexity. We illustrate in \Cref{sec:Comparison} how this upper-bound varies across several imputation procedures.


\begin{remark}
    The bound remains valid for $\lambda=0$. In this unregularized case, the predictor corresponds to the truncated OLS estimator. The bias term $B_0^{\mathcal I}$ vanishes, and the variance is governed by the raw dimension ratio $d/n$.
\end{remark}

\begin{remark}
    In practice, the bound parameter $m$ is unknown but can be estimated by $\hat m = \max_{i} |Y_i|$ \citep{mourtada2023local}. The theoretical guarantee provided by \Cref{thm:imputeechangeable} remains valid with $\hat m$, up to an additional negligible term of order $O(m^2/n)$.
\end{remark}

\subsection{Practical Federated Implementation}
We now address the practical implementation of the ItR strategy. In a federated context, the chosen imputation method must be computationally efficient and should not rely on sending any individual data. 

\textbf{Federated imputation.}
While recent works have proposed complex imputation approaches using federated GANs \citep{balelli2023fed}, we focus here on linear imputation methods. 
We categorize them by their communication cost, ranging from zero- to multi-shot procedures.

\noindent \textit{Zero-shot: naive imputation $\mathcal{I}^0$.} The simplest federated imputation strategies consist in replacing missing entries with a commonly-agreed constant, typically $0$ or the covariate means. Although this method introduces obvious bias, it cannot be ignored in a theoretical analysis as it represents the standard baseline for practitioners. Moreover, recent theoretical \citep{josse2019consistency,le2021sa,ayme2023naive,Van_Ness_2023} and empirical \citep{leimputation} studies have highlighted regimes where naive imputations yield competitive performance, particularly for high dimension. From a federated point of view, $\mathcal{I}^0$ is attractive because it entails zero communication cost.

\noindent\textit{One-shot: approximate $\mathcal I^\mathrm{opt}$.} To approximate the optimal linear imputation defined in \Cref{eq:optImpute} in federated learning, we can leverage the unbiased $\widehat\Sigma^{\mathrm{cw}}$ introduced in \Cref{subsec:PlugIn}. 
The imputation function becomes:
\begin{equation*}
    \tilde X_{i,\mis(k)} = (\widehat\Sigma^{\mathrm{cw}})_{\mis(k),\obs(k)}\,(\widehat\Sigma^{\mathrm{cw}})_{\obs(k)}^{\dagger}\,X_{i,\obs(k)} .
\end{equation*}
Importantly, since $\widehat\Sigma^{\mathrm{cw}}$ depends only on covariates, it can be estimated very accurately either from a large held-out unlabeled covariate dataset, or federatively via a one-shot aggregation (summing) of center-level second-moment statistics. This approach can substantially reduce the bias of naive imputation by exploiting cross-center correlations, at the cost of a single communication round of order $O(d^2)$ floats.

\noindent\textit{Multi-Shot: federated ICE.} 
When higher accuracy of the imputation is required, the popular Imputation by Chained Equations (ICE) framework \citep{van2011mice} can be implemented within the federated setting. Starting from an initial imputed dataset (e.g., via $\mathcal{I}^0$), the server repeatedly sums the empirical covariance of the current imputed data and broadcasts it; centers then update their missing coordinates by applying the corresponding recursive linear transformation,
\[
\tilde X_{i,\mis(k)}^{(t+1)}
\;=\;
(\widehat\Sigma_t)_{\mis(k),\obs(k)}\,
(\widehat\Sigma_t)_{\obs(k)}^{\dagger}\,X_{i,\obs(k)}.
\]
After $T$ rounds, the total communication cost scales as $O(Td^2)$. This approach should be used with caution in federated deployments, since repeatedly transmitting $\{\widehat\Sigma_t\}_{t\in[T]}$ can be communication-intensive and may increase privacy risks if not combined with appropriate protections (e.g., secure aggregation or differential privacy). Accordingly, we recommend this heuristic only when \(d\) is small and communication is not a bottleneck, or when the latter privacy protections are in place. 


\textbf{Federated Ridge regression.}
Once covariates have been imputed using one of the federated procedures above, the Ridge estimator in \Cref{eqdef:ridge_estimator} can also be learned in a federated manner.

\noindent\textit{One-Shot Estimation.}
The closed-form estimator $\widehat\theta_\lambda$ depends only on the aggregated sufficient statistics
$\widehat\Sigma^\mathcal{I}=\sum_{k=1}^K \frac{n_k}{n} \hat\Sigma_k^{\mathcal{I}}$ and $\widehat\gamma^\mathcal{I}=\sum_{k=1}^K \frac{n_k}{n} \hat\gamma_k^{\mathcal{I}}$.
These are summable statistics that can be computed exactly in a single communication round: each center $k$ forms the corresponding local sums on its imputed data and transmits them to the server, which aggregates and solves \eqref{eqdef:ridge_estimator}. 
This yields the {exact} ridge solution, but requires sending a $O(d(d+1))$ floats by center.

\noindent\textit{Multi-shot Estimation (FedAvg).}
When communicating $d\times d$ matrices is prohibitive, $\widehat\theta_\lambda$ can be obtained by minimizing the ridge objective iteratively,
\[
\widehat\theta_\lambda
=
\argmin_{\theta\in\mathbb R^d}
\Big\{\frac{1}{2n}\sum_{i=1}^n (Y_i-\theta^\top \widetilde X_i)^2+\frac{\lambda}{2}\|\theta\|_2^2\Big\}.
\]
This objective can be optimized with Federated Averaging (FedAvg) \citep{mcmahan2017communication}, which alternates local gradient steps on center datasets with global parameter averaging. This reduces the cost to $O(d)$ scalars exchanged per round, at the expense of requiring $T$ rounds. The solution converges to the exact one-shot estimator $\widehat \theta_\lambda$ as $T\to +\infty$. 
Consequently, this method is preferred for high-dimensional settings, especially as the $\ell_2$ penalization significantly accelerates the convergence \citep[see][for example]{dieuleveut2017harder}.

\section{Theoretical Comparison of Methods}\label{sec:Comparison}
In this section, we analyze the theoretical guarantees of \Cref{thm:imputeechangeable} for different federated imputation strategies, using local learning as a baseline. We benchmark our results against standard optimality criteria from the random design linear prediction literature. For the unregularized case ($\lambda=0$), we compare our rates to the minimax rate  for OLS \citep{tsybakov2003optimal,mourtada2019exact}. For the regularized setting ($\lambda>0$), we bound the bias and effective dimension terms using tight results established for high-dimensional Ridge regression \citep{caponnetto2007optimal,dicker2016ridge,richards2021asymptotics, ayme2025breaking}. 
 

\subsection{Local Learning}
To quantify the benefits of collaboration, we benchmark against a local learning baseline where centers do not communicate. Test points from center $k$ are predicted using a model trained solely on its local dataset. Formally, the local strategy is
\begin{equation*}
    \widehat f^{\mathrm{local}}_\lambda(X_{\obs(H)},H)=\sum_{k=1}^K \mathds{1}\{H=k\}\,\widehat f^{(k)}_{\lambda_k}(X_{\obs(k)}),
\end{equation*}
where $\widehat f^{(k)}_{\lambda_k}(x)=T_m\!\big(x^\top \widehat\theta^{(k)}_{\lambda_k}\big),$ and $\widehat\theta^{(k)}_{\lambda_k}=(\widehat\Sigma^{(k)}+\lambda_k I)^{-1}\widehat\gamma^{(k)}$ is computed from center $k$'s data only. We take $\lambda_k=\lambda/\rho_k$ to account for heterogeneous local sample sizes.

\begin{proposition}[Risk of local learning]\label{prop:local_predictor_gen}
\looseness=-1 Assuming~\Cref{ass:mcar_iid,ass:Ybound}, we have
\begin{equation*}
    E_0 \!\!\le\! R(\widehat f^{\mathrm{local}}_\lambda)\!-\!R_\star(\mathcal F_\mathrm{lin}) \!\le\! E_0 + \!\!\!\sum_{k=1}^K \!\rho_k B^{(k)}_{\lambda_k} + \frac{16m^2}{n}\!\!\!\sum_{k=1}^K d^{(k)}_{\lambda_k}
\end{equation*}
with \(E_0=\sum_{k} \rho_k(1-\rho_k)^n\,\E[Y^2]\) the error from centers with $0$ observations, and
\begin{equation*}
    \begin{array}{l}
         B^{(k)}_{\lambda_k} = \inf_{\theta} \Big\{\E\big[(Y-\theta^\top X_{\obs(k)})^2\big] - R_\star^{(k)} +\lambda_k\|\theta\|_2^2\Big\} \\
    d^{(k)}_{\lambda_k} = \mathrm{Tr}\Big(\Sigma_{\obs(k)}\big(\Sigma_{\obs(k)}+\lambda_k I\big)^{-1}\Big).
    \end{array}
\end{equation*}

\end{proposition}
This bound highlights the cost of fragmentation paid by local learning. First, when many centers have small \(\rho_k\), \(E_0\) may be non-negligible, reflecting that some centers effectively cannot train. Second, the variance term scales with \(\sum_k d^{(k)}_{\lambda_k}\) rather than the pooled effective dimension $d_\lambda$, reflecting the absence of statistical sharing across sites.

\subsection{Optimal Linear Imputation}

\begin{theorem}\label{thm:ImputOracleDB}
Under \Cref{ass:mcar_iid,ass:exchangeable,ass:Ybound}, the optimal imputation $\mathcal I^{\mathrm{opt}}$ satisfies 
\begin{equation*}
    \begin{array}{l}
         d_\lambda^{\mathcal I^{\mathrm{opt}}}\le d_\lambda:=\mathrm{Tr}\!\big(\Sigma(\Sigma+\lambda I_d)^{-1}\big),\\
        B_\lambda^{\mathcal I^{\mathrm{opt}}}\le B_\lambda:= \inf_{\theta\in\mathbb R^d}\Big\{\|\theta-\theta_\star\|_\Sigma^2+\lambda\|\theta\|_2^2\Big\}.         
    \end{array}
\end{equation*}
\end{theorem}
Thus, with \Cref{thm:imputeechangeable}, the risk of the associated ItR predictor is
\begin{align*}
        R(\hat f_\lambda^{\mathcal{I}^{\mathrm{opt}}})\leq R_\star(\mathcal F_\mathrm{lin})+ B_\lambda+ 8m^2\frac{d_\lambda}{n}. 
\end{align*}
\Cref{thm:ImputOracleDB} establishes that the ItR with optimal linear imputation converges to the global optimal risk \(R_\star(\mathcal F_{\mathrm{lin}})\) at a ``fast rate''. Moreover, these rates scale with the {total sample size} $n = \sum_{k} n_k$, rather than local sample sizes, which confirms the benefit of collaborative learning.
This spectral bound is particularly relevant in our high-dimensional setting as it captures the intrinsic dimensionality of the \emph{true} data. When features are highly correlated (e.g., due to aliasing across centers), $d_\lambda \ll d$, resulting in tighter bounds. 



\subsection{Naive Imputation}
We now compare local learning to the federated Ridge predictor built on zero-imputed covariates $\mathcal I^0$. 
\begin{proposition}\label{prop:0vsLocal}
Under \Cref{ass:mcar_iid}, 
\[
R_\star(\F^{\mathcal{I}^0})+B_\lambda^{\mathcal I^0}\;\ge\;R_\star(\mathcal F_\mathrm{lin})+\sum_{k=1}^K \rho_k B^{(k)}_{\lambda_k},
\]
and
\begin{equation*}
    d_\lambda^{\mathcal I^0}\;\le\;\sum_{k=1}^K d^{(k)}_{\lambda_k}.
\end{equation*}
\end{proposition}
This proposition highlights the bias-variance dilemma between the two strategies. Zero-imputation increases bias but lowers variance relative to local learning by replacing the sum of local effective dimensions with a smaller, single global dimension. For OLS ($\lambda=0$), this reduces the variance by a factor $K$, estimating $d$ parameters instead of $Kd$. Consequently, local learning is preferable when $K$ is small, whereas zero-imputation becomes competitive in highly fragmented regimes where variance dominates. We next quantify this imputation bias to identify when each strategy dominates based on dimension and missingness ratios.

\subsection{Typical-Case Analysis}\label{sec:typical}
The bias $B_\lambda^{\mathcal{I}^0}$ and approximation term $R_\star(\F_{\mathrm{lin}})$ in our bounds depend on the structure of the centers' mask \(\{\obs(k)\}_{k=1}^K\), which are difficult to interpret in full generality. We therefore analyze a \emph{typical-case} model where the observation patterns are randomized.

\textbf{Bernoulli observation model.}
We assume that prior to training, each center-feature pair is observed independently with probability $\tau\in(0,1]$:
\begin{equation}\label{eq:bernoulli_model}
    \mathds{1}\{j\in\obs(k)\}\sim\mathrm{Bernoulli}(\tau),    
\end{equation} with $k\in[K]$ and $j\in[d]$. Thus, $\tau$ quantifies the average data density, and $\tau=1$ corresponds to complete data.

\begin{proposition}\label{prop:typical_case}
Assume the features are normalized so that \(\mathrm{diag}(\Sigma)=I_d\). Under \Cref{ass:mcar_iid} and model \eqref{eq:bernoulli_model}, for any penalization \(\lambda\ge0\), the bias of the zero-imputed Ridge predictor satisfies
\begin{align*}
    \esp R_\star(\F_{\mathrm{lin}})
    &\leq \esp \big[ R_\star(\F^{\mathcal{I}^0}) + B_\lambda^{\mathcal{I}^0} \big] \\
    &\leq \sigma^2 + B_{\lambda'}
\end{align*}
where the expectation is over the random draw of \(\{\obs(k)\}_{k=1}^K\) and
\(
\lambda'
\;=\;
\frac{\lambda}{\tau^2}+\frac{1-\tau}{\tau}.
\)
\end{proposition}
This states that the approximation error $\esp R_\star(\F_{\mathrm{lin}})$ and the bias induced by zero-imputation are bounded by the classical Ridge bias of the complete data, controlled by the missingness rate $\lambda_0:=\frac{1-\tau}{\tau}$, and is thus negligible as soon $d\gg n\lambda_0$. This extends the results of \citet{ayme2023naive} to the block-wise missing value setting.

\textbf{Takeaway.}
This analysis allows us to settle the comparison between local learning and federated imputation. In high-dimensional settings, the ItR estimator with any imputation (whether Optimal or Naive) systematically outperforms local learning.


\section{Numerical experiments}\label{sec:numerical_experiments}
\begin{figure*}[!t]
    \centering
    \begin{subfigure}[t]{0.24\textwidth}
        \centering
        \includegraphics[width=\textwidth]{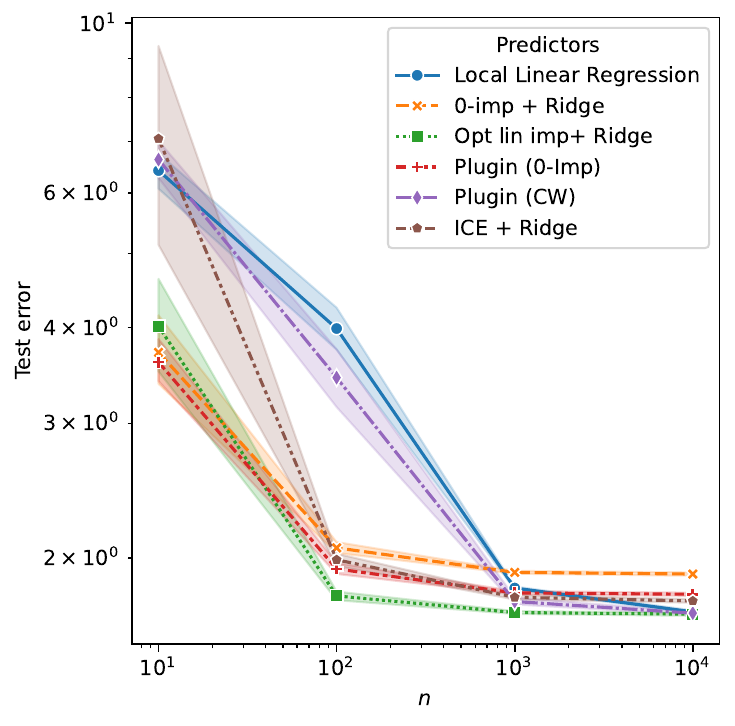}
        \caption{Fixed $K=50$, $d=8$.}
        \label{fig:lowdim}
    \end{subfigure}
    \hfill
    \begin{subfigure}[t]{0.24\textwidth}
        \centering
        \includegraphics[width=\textwidth]{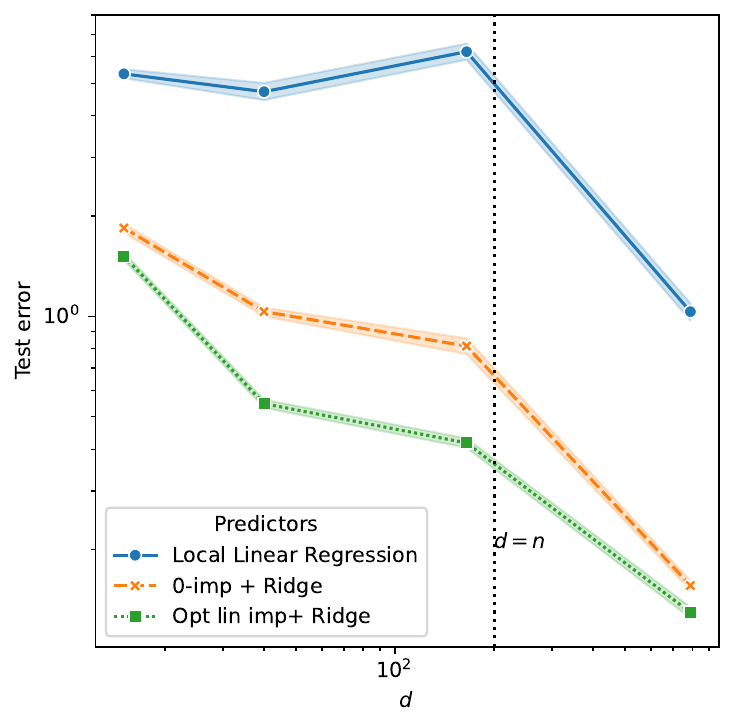}
        \caption{Fixed $K=50$, $n=200$.}
        \label{fig:highdim-d}
    \end{subfigure}
    \hfill
    \begin{subfigure}[t]{0.24\textwidth}
        \centering
        \includegraphics[width=\textwidth]{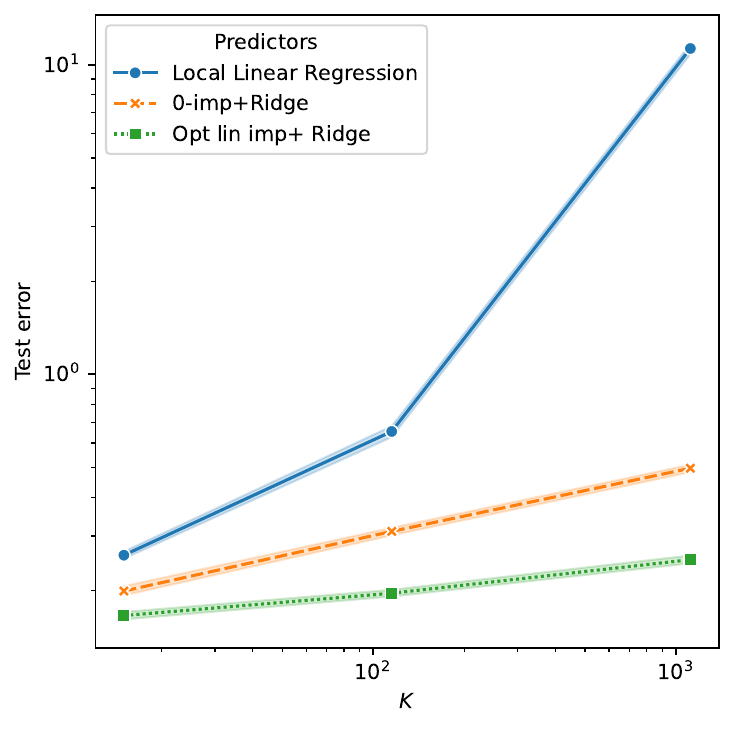}
        \caption{Fixed $n=1000$, $d=100$.}
        \label{fig:highdim-k}
    \end{subfigure}
    \hfill
    \begin{subfigure}[t]{0.24\textwidth}
        \centering
        \includegraphics[width=\textwidth]{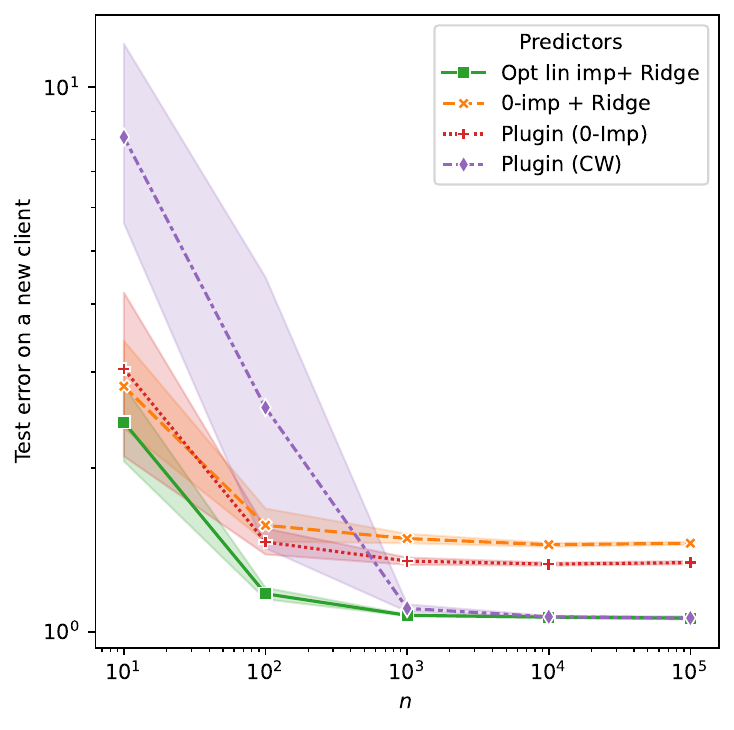}
        \caption{Test on an unseen center.}
        \label{fig:new_center}
    \end{subfigure}
    \caption{Prediction performance under covariate mismatch across four regimes.}
    \label{fig:simulations}
\end{figure*}

We generate data with a latent-factor regression model similar to the simulation setting considered in \citet{hastie2022surprises}. For each observation $i\in[n+1]$, $H_i\sim\mathcal{U}([K])$ and for feature $j\in[d]$,
$X_{ij} = W_j^\top Z_i + \Lambda \nu_{ij}$, and $Y_i = \beta^\top Z_i + \epsilon_i$,
where $Z_i \sim \mathcal{N}(0,I_p)$ is a $p$-dimensional latent variable, $W_j \sim \mathcal{N}(0,I_p)$, $\nu_{ij}\sim\mathcal{N}(0,1)$, and $\epsilon_i\sim\mathcal{N}(0,\sigma^2)$, independently. We set $p=3$, $\Lambda=1$, and $\sigma=0.3$ throughout. This model allows to vary both the dimension $d$ and the sample size $n$ while retaining a low-dimensional latent structure, which is convenient for studying the high-dimensional regime. Local sample sizes are balanced ($n_k\approx n/K$), and observed feature sets $\{\obs(k)\}_{k\in[K]}$ are generated according to model \eqref{eq:bernoulli_model} with $\tau=0.5$.

To ensure numerical stability, the predictions of the Plug-in estimators (both zero-imputed and CW) are clipped to the interval $[-10, 10]$. All Ridge predictors select their regularization parameter $\lambda$ via cross-validation. ICE imputation is implemented using \texttt{scikit-learn} with 3 iterations.

Figure~\ref{fig:simulations}\subref{fig:lowdim} illustrates the consistency of the CW plug-in estimator (\Cref{thm:consistency}). While the zero-imputed plug-in estimator improves performance for small $n$, it remains asymptotically biased (\Cref{subsec:PlugIn}). The ICE-then-regress predictor provides more accurate imputations, but its computational cost scales poorly with $d$, and a residual bias, lower than that of the zero-imputation, remains, highlighting the trade-off between efficiency and accuracy. All the proposed collaborative estimators outperform local learning in small-sample regimes. As local sample sizes $n_k$ grow, the benefit of collaboration naturally diminishes, and local estimators become increasingly competitive.

Figure~\ref{fig:simulations}\subref{fig:highdim-d} shows that even in mild dimensionalities, zero imputation $\mathcal I^0$ followed by Ridge regression outperforms local learning.
The risk associated with each plotted predictor decreases with $d$, as more features become available to reconstruct the latent variables $Z$.
Figure~\ref{fig:simulations}\subref{fig:highdim-k} shows this performance gap grows further as the number of centers $K$ increases: distributing the $n$ samples across more centers reduces local sample sizes $n_k$, which degrades local prediction, whereas the zero-imputed-then-regress predictor leverages the $n$ data points. However, an increasing $K$ introduces higher variance in the zero-imputed $\tilde X$. 

Finally, Figure~\ref{fig:simulations}\subref{fig:new_center} demonstrates the zero-shot generalization established in \Cref{thm:consistency}. Local learning cannot be evaluated on an unseen center $K+1$ lacking training data and labeled outcomes. In contrast, the plug-in CW estimator achieves consistency and converges to the true risk without requiring prior knowledge of the covariance $\Sigma$. The optimal imputation $\mathcal I^\mathrm{opt}$ yields the theoretical best predictor on this new center (\Cref{prop:best_linear}) and serves as an oracle benchmark across all regimes, illustrating the performance attainable if $\Sigma$ were known, or estimated from large unlabeled data. When data is scarce and $\Sigma$ is unknown, zero imputation combined with Ridge regression or the plug-in estimator remains a viable practical alternative for deploying to a new center.

\section{Discussion}
We have established a rigorous framework for performing federated linear regression under covariate mismatch. In low-dimensional settings, Plug-in estimators are preferred for their ability to generalize to new centers with unseen feature patterns. Conversely, in high-dimensional or fragmented settings, the Impute-then-Ridge regress strategy strictly outperforms local learning. By pooling data to stabilize the global effective dimension, it achieves optimal estimation rates that isolated local training cannot match.

Relaxing \Cref{ass:mcar_iid} poses a central challenge: handling covariate and outcome distribution shifts beyond local learning remains an open question, and achieving provable collaborative gains require more elaborate models.

A promising avenue is extending this framework beyond linear predictors. Because our guarantees depend on the \emph{effective} rather than nominal dimension, they naturally support using linear embeddings (e.g., Random Fourier Features \citep{rahimi2007random}) to construct over-parameterized non-linear models. This also motivates the integration of fully non-parametric models like random forests, which have recently been adapted to federated settings \citep{khellaf2026principled} and are robust to missing values \citep{stekhoven2012missforest}.



\bibliography{bib}
\bibliographystyle{plainnat}

\clearpage

\appendix

\onecolumn
\aistatstitle{Supplementary Materials}
\section{Optimal risk over the the linear class}
The global risk decomposition is given by the tower rule: 
\begin{align*}
    R(f) &= \E \left[ \left( Y - f^{(H)}([X]_{\obs(H)}) \right)^2 \right] \\
    &= \E \esp \left[ \left( Y - f^{(H)}([X]_{\obs(H)}) \right)^2 |H\right]\\
    &=\sum_{k=1}^K \rho_k \esp \left[ \left( Y - f^{(H)}([X]_{\obs(H)}) \right)^2 |H\right]\\
    &= \sum_{k=1}^K \rho_k R^{(k)}(f^{(k)}).
\end{align*}
Under \Cref{ass:mcar_iid}, we can remove the conditional dependence in $H$: 
\begin{equation*}
    R^{(k)}(f^{(k)})= \esp \left[ \left( Y - f^{(H)}([X]_{\obs(H)}) \right)^2 \right].
\end{equation*}

\begin{proof}[Proof of \Cref{prop:best_linear}]\label{proof:best_linear}
The complete model satisfies $Y=\theta^\top X+\epsilon$ with $\esp[\epsilon X]=0$. Then, denoting $S_k:= \Sigma_{\mis(k),\obs(k)}(\Sigma_{\obs(k)})^{-1}$, we have 
\begin{align*}
    Y&=\theta_{\obs(k)}^\top X_{\obs(k)}+ \theta_{\mis(k)}^\top (X_{\mis(k)}-S_kX_{\obs(k)}+S_kX_{\obs(k)})+\epsilon\\
    &=(\theta_{\obs(k)}+S_k^\top \theta_{\mis(k)})  ^\top X_{\obs(k)}+ \theta_{\mis(k)}^\top (X_{\mis(k)}-S_kX_{\obs(k)})+\epsilon.
\end{align*}
Denoting by $\epsilon'=\theta_{\mis(k)}^\top (X_{\mis(k)}-S_kX_{\obs(k)})+\epsilon $, we have 
\begin{align*}
    \esp [\epsilon' X_{\obs(k)}]&= \esp[ X_{\obs(k)} (X_{\mis(k)}-S_kX_{\obs(k)})^\top \theta_{\mis(k)}]+ 0\\
    &= (\Sigma_{\obs(k),\mis(k)}-\Sigma_{\obs(k)}^{-1}S_k^\top)\theta_{\mis(k)}\\
    &= 0.
\end{align*}
This shows that $\theta_{\obs(k)}+S_k^\top \theta_{\mis(k)}$ is the best linear approximation of $Y$ given $X_{\obs(k)}$, as the associated residual satisfies the first order condition $\esp[ \epsilon'X_{\obs(k)}]=0$ (see \Cref{lem:regressionLineaire}). 

Furthermore, 
\begin{align*}
    \esp [\epsilon'^{2}]&= \sigma^2+ \esp[\theta_{\mis(k)}^\top (X_{\mis(k)}-S_kX_{\obs(k)}) (X_{\mis(k)}-S_kX_{\obs(k)})^\top\theta_{\mis(k)}]\\
    &=  \sigma^2+ \theta_{\mis(k)}^\top \esp[ (X_{\mis(k)}-S_kX_{\obs(k)}) (X_{\mis(k)}-S_kX_{\obs(k)})^\top]\theta_{\mis(k)}\\
    &= \sigma^2+ \theta_{\mis(k)}^\top (\Sigma_{\mis(k)}-\Sigma_{\mis(k),\obs(k)}\Sigma_{\obs(k)}^{-1}\Sigma_{\obs(k),\mis(k)})\theta_{\mis(k)}
\end{align*}
Then $\esp [\epsilon'^{2}]= \Vert\theta_{\mis(k)}\Vert_{\Sigma_{\mis(k),\obs(k)}\Sigma_{\obs(k)}^{-1}\Sigma_{\obs(k),\mis(k)}}^2$, which completes the proof.

\end{proof}

\section{Proofs of Section \ref{subsec:PlugIn}}\label{app:PI}

\begin{lemma}[Linear prediction tools]\label{lem:regressionLineaire}
    Under the linear model $Y=\theta_\star^\top X +\varepsilon$ (well-specified or not), we have 
    \begin{itemize}
        \item Best linear predictor: $\theta_\star= (\esp XX^\top)^{-1}\esp[XY]=\Sigma^{-1}\gamma$.
        \item First order condition: $\esp [\epsilon X ]=0$.
        \item Comparison with the best linear model: 
        \begin{equation*}
            R(\theta)-R(\theta_\star)= \Vert \theta-\theta_\star\Vert_\Sigma^2
        \end{equation*}
        \item Considering an unbiased estimator $(\hat\Sigma,\hat \gamma)$ and assuming $\exists L>0$ such that $\Vert \theta_\star\Vert\leq L$, then the predictor 
        \begin{equation*}
            \hat \theta \in \arg \min_{\Vert \hat \theta\Vert\leq L}\{\Vert \theta\Vert_{\hat \Sigma}^2-2 \hat \gamma^\top \theta\},
        \end{equation*}
        satisfies 
            \begin{equation*}
                \esp R(\hat \theta)- R( \theta_\star)\leq L^2 \esp \Vert\Sigma-\hat \Sigma \Vert_{\mathrm{sp}}+ 2 L^2 \esp\Vert \gamma-\hat \gamma \Vert.
            \end{equation*}
    \end{itemize}
\end{lemma}
\begin{proof}
By optimality of $\theta_\star$, we have $\nabla R(\theta_\star) = 0$, thus,
\begin{equation}
    \esp [X (X^\top\theta_\star-Y)]=0,
\end{equation}
which is equivalent to $\Sigma\theta=\gamma$. This result also yields that $\epsilon= Y- X^\top\theta_\star$ satisfies 
\begin{equation*}
    \esp [\epsilon X ]=0.
\end{equation*}
Thus, 
\begin{align*}
    R(\theta)-R(\theta_\star)&= \esp [(Y-X^\top \theta)^2]-R(\theta_\star)\\
    &=\esp [(\epsilon-X^\top (\theta-\theta_\star))^2]-R(\theta_\star)\\
    &= \esp [\epsilon^2]-2\esp [\epsilon X^\top (\theta-\theta_\star)]+\esp [(X^\top (\theta-\theta_\star))^2]-R(\theta_\star)\\
    &=R(\theta_\star)+0+ \Vert \theta-\theta_\star\Vert_\Sigma^2 -R(\theta_\star)\\
    &= \Vert \theta-\theta_\star\Vert_\Sigma^2.
\end{align*}

Remark that, 
    \begin{align*}
        R(\theta)&= \esp[(\theta^\top X-Y)^2]\\
        &= \esp[(\theta^\top X)^2]-2\esp[(\theta^\top X)Y]+\esp[Y^2]\\
        &= \Vert \theta\Vert_{\Sigma}^2-2 \gamma^\top\theta + \esp[Y^2].
    \end{align*}
$\esp[Y^2]$ does not depend on $\theta$, thus we consider $R(\theta)= \Vert \theta\Vert_{\Sigma}^2-2 \gamma^\top\theta$.
We denote by $\hat R(\theta)= \Vert \theta\Vert_{\hat \Sigma}^2-2 \hat \gamma^\top \theta $, 
\begin{align*}
     R(\hat \theta)&= R(\hat \theta)-\hat R(\hat \theta)+\hat R(\hat \theta)\\
     &\leq R(\hat \theta)-\hat R(\hat \theta)+\hat R( \theta_\star).
\end{align*}
If $\hat\gamma$ and $\hat \Sigma$ are unbiased, we have by linearity $\esp \hat R( \theta_\star)=R( \theta_\star)$. Thus, we obtain 
\begin{equation*}
    \esp R(\hat \theta)- R( \theta_\star)\leq \esp[R(\hat \theta)-\hat R(\hat \theta)].
\end{equation*}
Considering it as a minimization problem under the constraint $\Vert\hat \theta\Vert\leq L$ yields
\begin{equation*}
    \esp R(\hat \theta)- R( \theta_\star)\leq \esp[\sup_{\Vert\hat \theta\Vert\leq L}\{R( \theta)-\hat R( \theta)\}].
\end{equation*}
Since $R( \theta)-\hat R( \theta)= \Vert \theta\Vert_{\Sigma-\hat \Sigma}^2-2 (\gamma-\hat \gamma)^\top \theta$, we obtain 
\begin{equation*}
    \esp R(\hat \theta)- R( \theta_\star)\leq L^2 \esp \Vert\Sigma-\hat \Sigma \Vert_{\mathrm{sp}}+ 2L^2 \esp\Vert \gamma-\hat \gamma \Vert.
\end{equation*}
\end{proof}

\begin{proof}[Proof of \Cref{thm:consistency}]\label{proof:thm_consistency}
\textbf{Consistency.}
    Let $j,j'\in[d]^2$ such that $\Pi_{jj'}>0$. As unbiased estimators, $\hat\gamma^{\mathrm{cw}}_{j},\hat\Sigma^{\mathrm{cw}}_{jj'}$ converge to the true $\gamma_j, \Sigma_{jj'}$. The consistency result is simply an application of the continuous mapping theorem: for the CW estimator, we observe that $N_{jj'}$, the number of observations where $j,j'$ are observed, satisfies $n/N_{jj'}\to 1/\Pi_{jj'}$ almost surely if $\Pi_{jj'}>0$. 

\textbf{Zero-shot generalization of CW plug-in.}
\begin{propo}\label{prop:non-ass} Assume that $\Vert \theta_\star^{(K+1)}\Vert\leq L$, $m^2=\max(\esp [\Vert X\Vert_2^4,\esp[Y^2\Vert X\Vert_2^2])<+\infty$ and $\pi=\min_{jj'\in\obs(K+1)} \Pi_{jj'}$, then the solution of 
\begin{equation*}
            \hat \theta^{(K+1)} \in \argmin_{\Vert \hat \theta\Vert\leq L}\{\Vert \theta\Vert_{\hat \Sigma^{\mathrm{cw}}_{\obs(K+1)}}^2-2 \theta^\top{\hat \gamma_{\obs(K+1)}^{\mathrm{cw}}}\},
        \end{equation*}
yields
    \begin{equation*}
         R^{(K+1)}(\hat \theta^{(K+1)})- R_\star^{(K+1)}(\mathcal{F}^{(K+1)}_{\mathrm{lin}})= O\left(\sqrt{\frac{d^2}{\pi n}}\right).
    \end{equation*}
\end{propo}

To prove this proposition, without loss of generality, we assume that we want to recover $\theta_\star$ and that $\pi= \min_{jj'\in[d]} \Pi_{jj'}>0$. Consider the following minimization problem. 
\begin{equation*}
            \hat \theta \in \arg \min_{\Vert \hat \theta\Vert\leq L}\{\Vert \theta\Vert_{\hat \Sigma^{\mathrm{cw}}}^2-2  \theta^\top {\hat \gamma^{\mathrm{cw}}}\},
        \end{equation*}
which, under \cref{lem:regressionLineaire}, satisfies
\begin{equation*}
                \esp R(\hat \theta)- R( \theta_\star)\leq L^2 \esp \Vert\Sigma-\hat \Sigma^{\mathrm{cw}} \Vert_{\mathrm{sp}}+ 2 L^2 \esp\Vert \gamma-\hat \gamma^{\mathrm{cw}} \Vert.
\end{equation*}
Bounding each term yields
\begin{align*}
    \esp \Vert\Sigma-\hat \Sigma^{\mathrm{cw}} \Vert_{\mathrm{sp}}&\leq \sqrt{\esp \Vert\Sigma-\hat \Sigma^{\mathrm{cw}} \Vert_{\mathrm{sp}}^2}\\
    &\leq \sqrt{\esp \Vert\Sigma-\hat \Sigma^{\mathrm{cw}} \Vert_{\mathrm{Fr}}^2},
\end{align*}
using Jensen inequality and  $\Vert\Vert_{\mathrm{sp}}^2\leq \Vert\Vert_{\mathrm{Fr}}^2$ (Frobenius norm). 
The Frobenius norm corresponds to the sum of the square of each component, thus we have 
\begin{align*}
    A_{jj'}&= (\esp [\Vert\Sigma-\hat \Sigma^{\mathrm{cw}} \Vert_{\mathrm{Fr}}^2|H_1,...,H_n])\\
    &\leq  \ind_{N_{jj'}=0} \Sigma_{jj'}^2+\ind_{N_{jj'}>0}\esp \left[ \left(\Sigma_{jj'}-1/N_{jj'}\sum_{i } \ind_{j,j'\in \obs(H_i)}X_{ij}X_{ij'}\right)^2\right]\\
    &\leq \ind_{N_{jj'}=0}\Sigma_{jj'}^2+\frac{\ind_{N_{jj'}>0}}{N_{jj'}}\esp[X_j^2X_{j'}^2].
\end{align*}
 Using that $N_{jj'}\sim \mathcal{B}(\nu,n)$ follows a binomial law with parameter $ \nu\leq\pi$, we have 
\begin{align*}
    \esp A_{jj'}&\leq(1-\pi)^n\Sigma_{jj'}^2+\frac{2}{n\pi}\esp[X_j^2X_{j'}^2].
\end{align*}
Summing over $j,j'$ gives us 
\begin{equation*}\label{eq:boundMatrix}
    \esp \Vert\Sigma-\hat \Sigma^{\mathrm{cw}} \Vert_{\mathrm{Fr}}^2\leq d^2(1-\pi)^n\mathrm{Tr}(\Sigma^2)+ \frac{2m^2 d^2}{n\pi},
\end{equation*}
because $ \sum_{jj'}\esp[X_j^2X_{j'}^2]= \esp [\Vert X\Vert^4]\leq m^2$. For the second term, we have, similarly,
\begin{equation*}
    \esp\Vert \gamma-\hat \gamma^{\mathrm{cw}} \Vert^2\leq d(1-\pi)^n\Vert \gamma\Vert_2^2+ \frac{2m^2d}{n\pi}.
\end{equation*}
Summing these two inequalities and using the triangle inequality,
\begin{equation*}
                \esp R(\hat \theta)- R( \theta_\star)\leq 6L^2 m \frac{d}{\sqrt{\pi n}}+ 2L^2 d(1-\pi)^{n/2}(\Vert \gamma\Vert_2+ \sqrt{\mathrm{Tr}(\Sigma^2)}).
\end{equation*}
Finally, with $\Vert \gamma\Vert_2^2+ \mathrm{Tr}(\Sigma^2)\leq m$ (Jensen inequality), we have 
\begin{equation*}
                \esp R(\hat \theta)- R( \theta_\star)\leq 6L^2 m \left(\frac{d}{\sqrt{\pi n}}+ (1-\pi)^{n/2}\right),
\end{equation*}
which concludes the proof. 

\end{proof}

\section{Proof of Section \ref{sec:ITR} and  \ref{sec:Comparison}}
\textbf{Prior result.}
\begin{theorem}[From \citet{mourtada2023local} theorem 4]\label{thm:jaouad}
    Assuming $(X_i,Y_i)_{i\in[n+1]}$ are exchangeable data, $\esp X_1X_1^\top=\Sigma$ exists and $|Y|\leq m$ almost surely,  For the predictor $\hat f(x) = T_m(\hat\theta_\lambda^\top x)$, with $\hat\theta_\lambda$ the ridge predictor learned on data $(X_i,Y_i)_{i\in[n]}$, we have 
    \begin{equation*}
        \esp[(Y_{n+1}-\hat f ( X_{n+1}))^2]\leq \inf_{\theta\in \R^d} \left\{\esp[(Y_{n+1}-\theta^\top X_{n+1})^2]+ \lambda \Vert\theta\Vert^2_2  \right\} + \frac{8m^2}{n}\mathrm{Tr}(\Sigma(\Sigma +\lambda I)^{-1}).
    \end{equation*}
\end{theorem}
The proof, provided by \citet{mourtada2023local}, only makes use of exchangeability arguments. 


\begin{proof}[Proof of \Cref{thm:imputeechangeable}]\label{proof:thm_imputeechangeable}
The exchangeability of the imputation function leads to the exchangeability of the imputed data $(\tilde X_i, Y_i)_{i\in [n+1]}$. Thus 
    \begin{equation*}
        \esp[(Y_{n+1}-\hat f_\lambda^{\mathcal{I}} ( \tilde X_{n+1}))^2]\leq \inf_{\theta\in \R^d} \left\{\esp[(Y_{n+1}-\theta^\top \tilde X_{n+1})^2]+ \lambda \Vert\theta\Vert^2_2  \right\} + \frac{8m^2}{n}\mathrm{Tr}(\Sigma^{\mathcal{I}}(\Sigma^{\mathcal{I}} +\lambda I)^{-1}).
    \end{equation*}
    where $\Sigma^{\mathcal{I}}=\esp \tilde X_1\tilde X_1^\top $ exists. Furthermore, considering $\theta'$ as the minimizer of $\theta \to \esp[(Y_{n+1}-\theta^\top \tilde X_{n+1})^2]$, the first order condition gives us 
    \begin{equation*}
        \esp[(Y_{n+1}-\theta'^\top \tilde X_{n+1})X_{n+1}]=0.
    \end{equation*}
    Then 
    \begin{align*}
        \esp[(Y_{n+1}-\theta^\top \tilde X_{n+1})^2]&= \esp[(Y_{n+1}-\theta'^\top \tilde X_{n+1}+ (\theta'-\theta)^\top \tilde X_{n+1})^2]\\
        &= \esp[(Y_{n+1}-\theta'^\top \tilde X_{n+1})^2]+\esp[( (\theta'-\theta)^\top \tilde X_{n+1})^2]\\
        &= R_\star(\F^{\mathcal{I}})+ \Vert \theta'-\theta\Vert_{\Sigma^{\mathcal{I}}}^2,
    \end{align*}
    which leads to the desired result, with 
    \begin{equation}\label{eq:biasI}
        B_\lambda^{\mathcal{I}}= \inf_{\theta}\{\Vert \theta'-\theta\Vert_{\Sigma^{\mathcal{I}}}^2+\lambda\Vert \theta\Vert_{2}^2\}.
    \end{equation}

\end{proof}

\textbf{Local predictors.}
\begin{proof}[Proof of \cref{prop:local_predictor_gen}]
    Starting with the decomposition, 
    \begin{equation}\label{eq:decompo_proof_local}
        R(\hat f)= \sum_k \rho_k \esp[(Y-\hat f_{\lambda_k}^{k} (X_{\obs(k)}))^2].
    \end{equation}
    and denoting $N_k$ the random integer corresponding to the number observations in center $k$, we have, using \Cref{thm:jaouad} and considering that $f_{\lambda_k}^{k} (X_{\obs(k)})=0$ if $N_k=0$
    \begin{align*}
        \ind_{N_k=0}\esp Y^2&\leq \esp[(Y-\hat f_{\lambda_k}^{k} (X_{\obs(k)}))^2|N_k]\\
        &\leq \ind_{N_k=0}\esp Y^2+ R_\star^{(k)}(\F^{(k)}_\mathrm{lin})+B^{(k)}_{\lambda_k}+ \frac{8 \ind_{N_k>0}m^2}{N_k} d^{k}_{\lambda_k}.
    \end{align*}
Taking the expectation to this inequality, using $\esp \ind_{N_k=0}=(1-\rho_k)^n$ and $\esp [\ind_{N_k=0}/N_k]\leq 2/(n\rho_k)$ (properties of the Binomial law), we have 
\begin{align*}
        (1-\rho_k)^n\esp Y^2&\leq \esp[(Y-\hat f_{\lambda_k}^{k} (X_{\obs(k)}))^2|N_k]\\
        &\leq (1-\rho_k)^n\esp Y^2+ R_\star^{(k)}(\F^{(k)}_\mathrm{lin})+B^{(k)}_{\lambda_k}+ \frac{16 m^2}{n\rho_k} d^{k}_{\lambda_k} .
    \end{align*}
We conclude using \eqref{eq:decompo_proof_local}.
\end{proof}

\textbf{Comparison naive imputation vs loca learning.}
\begin{proof}[Proof of \Cref{prop:0vsLocal}]
    Marginalizing over the centers gives 
    \begin{align*}
            \esp [(Y-\tilde X^\top \theta_\star)^2]&=\esp [(Y- S_HX^\top \theta_\star)^2]\\
            &= \esp [(Y- X^\top S_H \theta_\star)^2]\\
            &=\sum_k \rho_k \esp [(Y-\tilde X^\top \theta)^2]\\
            &=\sum_k \rho_k R^{(k)}_\star(\F^{(k)}_{\mathrm{lin}})\\
            &= R_\star(\F_{\mathrm{lin}}).
    \end{align*}
    Thus $R_\star(\F_{\mathrm{lin}})\geq R_\star(\F^{I^\mathrm{opt}})$. Remark that the functions in $\F^{I^\mathrm{opt}}$ are linear in $H$, thus $R_\star(\F_{\mathrm{lin}})= R_\star(\F^{I^\mathrm{opt}})$ and $\theta_\star$ attains the minimum. 
\end{proof}

\textbf{Comparison optimal imputation vs local learning.}
\begin{proof}[Proof of \Cref{thm:ImputOracleDB}]
\textbf{Step 1:} let us start by showing that $\Sigma^{{I^\mathrm{opt}}}\preceq \Sigma$: 
    \begin{align*}
\Sigma_k^{\mathcal{I}^{\mathrm{opt}}}
    &=\esp[ \tilde X \tilde X^\top|H=k]\\
    &= \begin{pmatrix}
    \Sigma_{\obs(k)} & \Sigma_{\obs(k),\mis(k)} \\
    \Sigma_{\mis(k),\obs(k)} & \Sigma_{\mis(k),\obs(k)}\Sigma_{\obs(k)}^{-1}\Sigma_{\obs(k),\mis(k)}
\end{pmatrix}.
\end{align*}
We remark that 
\begin{align*}
    \Sigma-\Sigma_k^{\mathcal{I}^{\mathrm{opt}}}= \begin{pmatrix}
    0 & 0 \\
    0 & I-\Sigma_{\mis(k),\obs(k)}\Sigma_{\obs(k)}^{-1}\Sigma_{\obs(k),\mis(k)}
\end{pmatrix}\succeq0,
\end{align*}
using the Schur complement properties. Then, $\Sigma_k^{\mathcal{I}^{\mathrm{opt}}}\preceq \Sigma$ which implies that $\Sigma^{\mathcal{I}^{\mathrm{opt}}}=\sum_{k=1}^K \rho_k \Sigma_k^{\mathcal{I}^{\mathrm{opt}}}\preceq \Sigma$.

\textbf{Step 2: degree of freedom.} Using the monotonicity of the Trace operator $A\to \mathrm{df}(A,\lambda):=\mathrm{Tr}(A(A+\lambda I)^{-1})$, for all $\lambda>0$, $$\mathrm{df}(\Sigma^{\mathcal{I}^{\mathrm{opt}}},\lambda)\leq \mathrm{df}(\Sigma,\lambda).$$

\textbf{Step 3: bias term.} We have, using \Cref{eq:biasI} with $\theta'=\theta_\star$ (the minimizer after optimal imputation) 
\begin{align*}
    B_\lambda^{\mathcal{I}^{\mathrm{opt}}}&= \inf_{\theta}\{\Vert \theta_\star-\theta\Vert_{\Sigma^{\mathcal{I}^{\mathrm{opt}}}}^2+\lambda\Vert \theta\Vert_{2}^2\}\\
    &\leq \inf_{\theta}\{\Vert \theta_\star-\theta\Vert_{\Sigma}^2+\lambda\Vert \theta\Vert_{2}^2\},
\end{align*}
using $\Sigma^{{I^\mathrm{opt}}}\preceq \Sigma$. 

\end{proof}

\begin{proof}[Proof of \Cref{prop:0vsLocal}]
Considering the imputation by $0$, for the bias, we have for all $\theta$,
\begin{align*}
    \esp[(Y- \tilde X^\top \theta)^2]+\lambda\Vert\theta\Vert_2^2&= \sum_k \rho_k\esp[(Y-  X_{\obs(k)}^\top \theta)^2]+\lambda\Vert\theta\Vert_2^2\\
    &= \sum_k \rho_k\left(\esp[(Y-  X_{\obs(k)}^\top \theta)^2]+\lambda/\rho_k\Vert\theta\Vert_2^2\right)\\
    &= \sum_k \rho_k\left(\esp[(Y-  X_{\obs(k)}^\top \theta)^2]+\lambda_k\Vert\theta\Vert_2^2\right)\\
    &\geq \sum_k \rho_k \inf_{\theta}\left\{\esp[(Y-  X_{\obs(k)}^\top \theta)^2]+\lambda_k\Vert\theta\Vert_2^2\right\}\\
    &\geq \sum_k \rho_k(R_\star^{(k)}(\F^{(k)}_\mathrm{lin})+B^{(k)}_{\lambda_k}).
\end{align*}
Then, 
\begin{equation*}
    B_\lambda^{\mathcal{I}^0}+ R_\star(\F^{\mathcal{I}^0})\geq   \sum_k \rho_k(R_\star^{(k)}(\F^{(k)}_\mathrm{lin})+B^{(k)}_{\lambda_k}).
\end{equation*}
For the effective dimension, we remark that 
\begin{equation*}
    \Sigma^{\mathcal{I}^0}= \sum_k \rho_k \Sigma_k, 
\end{equation*}
where $\Sigma_k $ is the covariance matrix of $X$ masked by $\obs (k)$. The lower bound is obtained using the Jensen inequality on $A\to \mathrm{Tr}(A(A+\lambda I)^{-1})$ which is concave over the set of non-negative matrices. For the upper bound, we compute, 
\begin{align*}
    \mathrm{Tr}(\Sigma^{\mathcal{I}^0}(\Sigma^{\mathcal{I}^0}+\lambda I)^{-1})&= \sum_k \rho_k \mathrm{Tr}(\Sigma_k(\Sigma^{\mathcal{I}^0}+\lambda I)^{-1})\\
    &\leq \sum_k  \mathrm{Tr}(\rho_k\Sigma_k(\rho_k\Sigma_k+\lambda I)^{-1})\\
    &= \sum_k  \mathrm{Tr}(\Sigma_k(\Sigma_k+\lambda/\rho_k I)^{-1})\\
    & = \sum_k d_{\lambda_k}^{(k)},
\end{align*}
using $\rho_k\Sigma_k\preceq \Sigma^{\mathcal{I}^0} $ and the monotony of the inverse.

\end{proof}

\section{Typical case}

\begin{proof}[Proof of \Cref{prop:typical_case}]
    We consider the imputation by $0$ . Let us start by computing $\esp[(Y- \tilde X^\top \theta)^2]$ in the typical case (integrated over $P= (P_{jk})_{j\in[d],k\in[K]}$). 
\begin{align*}
    \esp_P \esp[(Y- \tilde X^\top \theta)^2] &= \esp _P \esp[(Y- \tilde X^\top \theta)^2]\\
    &= \sum \rho_k \esp _P\esp[(Y-  X_{\obs(k)}^\top \theta)^2]\\
    &= \sum \rho_k \esp\esp _P[(Y-  X_{\obs(k)}^\top \theta)^2],
\end{align*}
using Fubini theorem on non-negative random variables. 
Furthermore, denoting $P_k = (P_{jk})_{j\in[d]}$, we have 
\begin{align*}
    \esp_P[(Y-  X_{\obs(k)}^\top \theta)^2]&= \esp_{P_k}[(Y-  X^\top (P_k\odot\theta))^2]\\
    &= \esp_{P_k}[(Y-  X^\top (\esp P_k\odot\theta))^2]+ \esp_{P_k}[X^\top (P_k-\esp P_k)\odot\theta))^2]\\
    &= \esp_{P_k}[(Y-  X^\top (\esp P_k\odot\theta))^2]+ \Vert \theta\Vert_{(\esp (P_k-\esp P_k)(P_k-\esp P_k)^\top)\odot XX^\top }^2.
\end{align*}
We have $\esp P_k= \tau U$, with $U=(1)_{j\in[d]}$, and 
\begin{align*}
    \esp (P_k-\esp P_k)(P_k-\esp P_k)^\top &= \esp P_kP_k^\top - \tau^2 UU^\top\\
    &= \tau^2 UU^\top + \tau(1-\tau) I - \tau^2 UU^\top\\
    &= \tau(1-\tau) I,
\end{align*}
since $(\esp P_kP_k^\top)_{jl}= \tau^2 \ind_{j\neq l}+ \tau \ind_{j= l}$. Thus, 
\begin{align*}
    \esp_P[(Y-  X_{\obs(k)}^\top \theta)^2]= \esp_{P_k}[(Y-  X^\top (\tau\theta))^2]+ \Vert \theta\Vert_{(\tau(1-\tau) I))\odot XX^\top }^2.
\end{align*}
Then, 
\begin{align*}
    \esp_P \esp[(Y- \tilde X^\top \theta)^2]&= \esp [(Y-  X^\top (\tau\theta))^2]+ \Vert \theta\Vert_{(\tau(1-\tau) I))\odot \Sigma }^2\\
    &= \esp [(Y-  X^\top (\tau\theta))^2]+ \Vert \theta\Vert_{\tau(1-\tau)\mathrm{diag}(\Sigma) }^2\\
    &= \esp [(Y-  X^\top (\tau\theta))^2]+ \tau(1-\tau)\Vert \theta\Vert_{2 }^2,
\end{align*}
using $\mathrm{diag}(\Sigma)=I$. Then, 
\begin{equation}\label{eq:integratedrisk}
    \esp_P \esp[(Y- \tilde X^\top \theta)^2]= \esp [(Y-  X^\top (\tau\theta))^2]+ \tau(1-\tau)\Vert \theta\Vert_{2 }^2.
\end{equation}
Thus, we consider $\tau \theta' =\argmin_{\theta}\{\esp [(Y-  X^\top \theta)^2]+ \lambda'\Vert \theta\Vert_{2 }^2\}$ (not random) with $\lambda'=\frac{\lambda}{\tau^2}+ \frac{1-\tau}{\tau}$.
We have by definition, 
\begin{equation*}
    B_\lambda^{\mathcal{I}^0}+ R_\star(\F^{\mathcal{I}^0})\leq \esp[(Y- \tilde X^\top \theta')^2|P]+ \lambda\Vert \theta'\Vert_{2 }^2,
\end{equation*}
almost surely. Thus,
\begin{equation*}
    \esp B_\lambda^{\mathcal{I}^0}+ \esp R_\star(\F^{\mathcal{I}^0})\leq \esp[(Y- \tilde X^\top \theta')^2]+ \lambda\Vert \theta'\Vert_{2 }^2. 
\end{equation*}
Finally, using \eqref{eq:integratedrisk}, 
\begin{align*}
    \esp B_\lambda^{\mathcal{I}^0}+ \esp R_\star(\F^{\mathcal{I}^0})&\leq \esp[(Y-  X^\top (\tau \theta'))^2]+ (\lambda+\tau(1-\tau))\Vert \theta'\Vert_{2 }^2\\
    &=\esp[(Y-  X^\top (\tau \theta'))^2]+ \lambda'\Vert \tau\theta'\Vert_{2 }^2\\
    &= \inf_\theta \{\esp [(Y-  X^\top \theta)^2]+ \lambda'\Vert \theta\Vert_{2 }^2\},
\end{align*}
by definition of $\theta'$. We conclude using $R_\star(\F_{\mathrm{lin}})\leq R_\star(\F^{\mathcal{I}^0})$.

\end{proof}

\thispagestyle{empty}

\end{document}